\documentclass[12pt, oneside,reqno]{amsart}

\usepackage[colorlinks=true,urlcolor=blue]{hyperref}

\usepackage{amsmath, amsfonts, amssymb, amsthm}
\usepackage{lastpage}

\usepackage{geometry}

\usepackage[pdftex]{graphicx}
\graphicspath{{pictures/}}

\usepackage{pgf,tikz}
\usetikzlibrary{arrows}

\DeclareMathOperator{\arctanh}{arctanh}

\theoremstyle{theorem}
\newtheorem{theorem}{Theorem}
\newtheorem*{conjecture}{Conjecture}
\newtheorem{lemma}{Lemma}

\theoremstyle{definition}
\newtheorem{definition}{Definition}
\newtheorem{remark}{Remark}

\DeclareMathOperator{\dist}{dist}
\DeclareMathOperator{\CAT}{CAT}
\DeclareMathOperator{\CBB}{CBB}

\usepackage{caption}
\DeclareCaptionLabelSeparator{dot}{. }
\title[Inradius estimates for convex domains in Alexandrov spaces]{Inradius estimates for convex domains in 2-dimensional Alexandrov spaces}

\author{Kostiantyn Drach}
\address{Jacobs University Bremen, Research I, Campus Ring 1, 28759 Bremen, Germany}
\email{k.drach@jacobs-university.de}

\begin{document}

\maketitle

\begin{abstract}
We obtain sharp lower bounds on the radii of inscribed balls for strictly convex isoperimetric domains lying in a 2-dimensional Alexandrov metric space of curvature bounded below. We also characterize the case when such bounds are attained.

\end{abstract}

\section*{Introduction and statement of the main results}

In this paper we address a reverse isoperimetric-type question, namely how \textit{small} the radius of the inscribed ball (\textit{inradius}) of a given domain can be. A direct question about the \textit{largest} inradius is rather trivial (for example, among all domains $D \subset \mathbb R^n$ of given volume only a ball has the largest inscribed ball, being the domain itself). Therefore, to make a reverse question meaningful and avoid trivial answers we have to impose some restrictions on the geometry of the domain.

One of the natural ways to do this is to assume some curvature conditions. Chakerian, Johnson and Vogt \cite{CJV76} obtained a sharp upper bound on the radius of the circumscribed ball for closed plane curves with curvature $|k| \leqslant 1$ (in a weak sense). Milka \cite{Mi78} substantially extended their result to curves with the same curvature restriction lying in $\mathbb R^n$, $n \geqslant 2$. Later on Alexander and Bishop \cite{AB97}, by using comparison techniques, transferred the results of Chakerian et al. and of Milka to $\CAT(\kappa)$ spaces. The goal of this note is to obtain lower bounds for radii of inscribed balls in Alexandrov spaces with lower curvature bounds, that is to prove a result \textit{dual} to that of Alexander--Bishop.

Let us also mention that the present work was motivated by the results concerning so-called \textit{reverse isoperimetric inequalities}. Particularly, Howard and Treibergs \cite{HTr95} proved a sharp reverse isoperimetric inequality on the Euclidean plane for closed embedded curves whose curvature $k$, in a weak sense, satisfies $|k| \leqslant 1$, and whose length is in $[2\pi, 14\pi/3)$ (see \cite[Theorem 4.1]{HTr95}). A dual result was obtained in all constant curvature spaces by Borisenko and the author in the series of papers \cite{BorDr14, BorDr15_1, Dr14}, where a two-dimensional reverse isoperimetric inequality was proved for so-called \textit{$\lambda$-convex} curves, i.e. curves whose curvature $k$, in a weak sense, satisfies $k \geqslant \lambda > 0$ (see Definition~\ref{Def:lconv} below) in constant curvature spaces. Recently, these results were generalized in \cite{BorNA} for $\lambda$-convex curves in Alexandrov metric spaces of curvature bounded below. We will use some of the results from \cite{BorNA} in the present paper.


Before stating the main result, let us set up some background and fix notation. For an extensive treatment of the theory of metric spaces, in particular, metric spaces of bounded curvature (Alexandrov spaces) we refer the reader to \cite{Al48} and \cite{BBI01}. 

Let $M$ be a geodesic metric space. Denote by $|pq|_M$ the distance between any two points $p, q \in M$. For a closed compact domain $D \subset M$ an \textit{inscribed ball} (or \textit{inball}, for short) is a largest ball contained in $D$. Hence the radius $r$ of an inball (\textit{inradius}) is given by 
\[
r = \max_{p \in D} \min_{q \in \partial D} |pq|_M.
\]

Denote by $M^n(\kappa)$ the $n$-dimensional \textit{model space of curvature $\kappa$}, that is the $n$-sphere $\mathbb S^n(k^2)$ of radius $1/k$ for $\kappa = k^2 > 0$, the Euclidean space $\mathbb R^n$ for $\kappa = 0$, and the hyperbolic space $\mathbb H^n(-k^2)$ for $\kappa = -k^2 < 0$.

For a triple of points $p,q,r \in M$ and the corresponding geodesic segments $pq$, $qr$, $rp$ (forming the \textit{triangle} $\triangle pqr$ in M) we associate a triangle $\triangle \tilde p \tilde q \tilde r$ in $M^2(\kappa)$ such that $|pq|_M = |\tilde p \tilde q|_{M^2(\kappa)}$, $|qr|_M = |\tilde q \tilde r|_{M^2(\kappa)}$, $|rp|_M = |\tilde r \tilde p|_{M^2(\kappa)}$. The latter triangle is called a \textit{comparison triangle}. A complete geodesic metric space $M$ is called an \textit{Alexandrov space with curvature $\geqslant\!\kappa$} (and abbreviated as $\CBB(\kappa)$) if for the angles of every triangle $\triangle pqr$ in $M$ the corresponding angles of a comparison triangle $\triangle \tilde p \tilde q \tilde r$ satisfy
\[
\angle pqr \geqslant \angle \tilde p \tilde q \tilde r, 
\quad \angle qrp \geqslant \angle \tilde q \tilde r \tilde p, \quad \angle rpq \geqslant \angle \tilde r \tilde p \tilde q, 
\]
and the sum of adjacent angles in $M$ equals $\pi$ (see  \cite{Al48, BBI01, BGP92} for details).

Further below we will say that a set $D$ in a metric space $M$ is a \emph{closed topological disk} if $D$ is closed, the boundary $\partial D$ of $D$ is a Jordan curve, and the interior of $D$ is homeomorphic to a disk. Most of the time we will be working in the setting when $\partial D$ is a `nice' curve (e.g.\,piecewise smooth).

Let $M$ be a two-dimensional $\CBB(\kappa)$ space, and $D \subset M$ be a closed topological disk with the rectifiable boundary curve $\gamma$. For every subarc $\widehat \gamma$ of $\gamma$ we can define the integral geodesic curvature (or the \textit{swerve}) of this arc as follows (see \cite[p.\,309]{Al48} for details). Suppose $\sigma := p_0p_1 \ldots p_Np_{N+1}$ is a broken geodesic line with vertices $p_i \in \widehat \gamma$ and such that $p_0$ and $p_{N+1}$ are the endpoints of $\widehat \gamma$. For each $i \in \{1, \ldots, N\}$ let $\alpha_i$ be the angle between geodesics $p_{i-1}p_i$ and $p_i p_{i+1}$ measured from the side of $D$, and let $\alpha_0$ (respectively, $\alpha_{N+1}$) be the angle between $\widehat \gamma$ and $p_0p_1$ (resp.\,$p_N p_{N+1}$) at $p_0$ (resp.\,$p_{N+1}$) measured from the side of $D$. Then the \emph{integral geodesic curvature} $\varphi(\widehat \gamma)$ of $\widehat \gamma$ (with respect to $D$) is defined as
\[
\varphi(\widehat\gamma):=\lim_{\sigma \to \widehat \gamma} \sum_{i=0}^{N+1} (\pi - \alpha_i).
\]  
It can be shown that the swerve of any subarc in our setting is well defined.

\begin{definition}[$\lambda$-convex domain]
\label{Def:lconv}
Let $M$ be a two-dimensional $\CBB(\kappa)$ space. A closed topological disk $D \subset M$ with the rectifiable boundary curve $\gamma$ is called a \textit{$\lambda$-convex domain} (with $\lambda>0$) if for every subarc $\widehat \gamma$ of $\gamma$
\begin{equation}
\label{Eq:swerve}
\varphi(\widehat \gamma) \geqslant \lambda \cdot s(\widehat \gamma),
\end{equation}
where $s(\widehat \gamma)$ is the length of the arc $\widehat \gamma$. 
\end{definition}

For domains with smooth boundary in two-dimensional Riemannian manifolds condition~(\ref{Eq:swerve}) is equivalent to the assumption that the geodesic curvature of the boundary at each point is at least $\lambda$. In general, $\lambda$-convex domains may have non-smooth points (such as corners) on the boundary. 

Recall that for a pair of convex curves $\gamma_1$ and $\gamma_2$ in $M^2(\kappa)$ intersecting at some point $s \in \gamma_1 \cap \gamma_2$ we say that $\gamma_2$ is \emph{locally supporting to $\gamma_1$ at $s$} if there is an open neighborhood $U$ of $s$ such that $D_1 \cap U \subset D_2 \cap U$, where $D_1$ (respectively, $D_2$) is a convex set with $\gamma_1 \subset \partial D_1$ (resp.\,$\gamma_2 \subset \partial D_2$). Observe that if $\gamma_1$ has a well-defined tangent geodesic at $s$, then $\gamma_2$ is locally supporting to $\gamma_1$ at $s$ if and only if $\gamma_1$ and $\gamma_2$ are tangent at $s$. The notion of local support provides yet another equivalent way of looking at $\lambda$-convexity in model spaces: a closed topological disk $D \subset M^2(\kappa)$ is $\lambda$-convex if it has a locally supporting curve of curvature $\lambda$ at every boundary point. This allows a generalization of $\lambda$-convexity to higher dimensions, see~\cite{Bor02, BorDr13}.

\begin{definition}[$\lambda$-convex lune]
\label{Def:LLune}
A \textit{$\lambda$-convex lune} in the model space $M^2(\kappa)$ is a compact convex region enclosed by two arcs of equal length and of constant geodesic curvature $\lambda$. 
\end{definition}

Observe that a lune is a centrally symmetric set, and hence there is a well-defined \emph{center of a $\lambda$-convex lune}.
 
From the well-known classification of curves of constant non-zero geodesic curvature it follows that for $\kappa \geqslant 0$ a $\lambda$-convex lune is enclosed by two circular arcs. The same is true for $\kappa < 0$ and $\lambda > \sqrt{-\kappa}$. However, if $\kappa < 0$ and $\lambda = \sqrt{-\kappa}$, then the lune is enclosed by two arcs of horocycles. Finally, if $\kappa < 0$ and $0<\lambda < \sqrt{-\kappa}$, then the boundary of the lune is composed of two arcs of equidistant curves (hypercycles) of curvature $\lambda$.  

Let us mention that $\lambda$-convex lunes can be constructed in a different way as follows. Suppose $F_\lambda \subset M^2(\kappa)$ is a closed convex set enclosed by a complete curve of constant geodesic curvature $\lambda$. For $\kappa \geqslant 0$, or for $\kappa < 0$ and $\lambda > \sqrt{-\kappa}$ the set $F_\lambda$ is just a closed geodesic disk, and hence $F_\lambda$ is compact. At the same time, in the hyperbolic case ($\kappa < 0$) for $\lambda \in \left(0, \sqrt{-\kappa}\right]$ the set $F_\lambda$ is unbounded: in the Poincar\'{e} model of the hyperbolic plane in the unit disk $F_\lambda$ is a closed Euclidean disk touching or intersecting the unit circle depending on whether, respectively, $\lambda=\sqrt{-\kappa}$ or $\lambda \in \left(0, \sqrt{-\kappa}\right)$. With this, every $\lambda$-convex lune can be constructed as an intersection of two regions of the form $F_\lambda$, and conversely, every pair of intersecting regions of the form $F_\lambda$ produces the $\lambda$-convex lune as their intersection. This is an equivalent way, compared to Definition~\ref{Def:LLune}, of defining a $\lambda$-convex lune. (In the simplest case, a lune is just the intersection of two geodesic disks.)

Observe that for given $\kappa$ and $\lambda > 0$, a $\lambda$-convex lune in $M^2(\kappa)$ is completely determined (up to isometry) by the length of its boundary. Hence, we have a well-defined function
\[
\rho_\lambda \colon I_\lambda \to \mathbb R_{+}, \quad L \mapsto \rho_\lambda(L),
\]
such that $\rho_\lambda(L)$ is the inradius of the $\lambda$-convex lune with the boundary of length $L$ in $M^2(\kappa)$. Here $I_\lambda$ is the natural domain of definition of this function: for $\kappa \geqslant 0$, or $\kappa < 0$ and $\lambda > \sqrt{-\kappa}$ we have $I_\lambda = \left[0, L_\lambda\right]$, where $L_\lambda$ is the length of a circle $\partial F_\lambda$; at the same time, $I_\lambda = [0, +\infty)$ for $\kappa < 0$ and $0 < \lambda \leqslant \sqrt{-\kappa}$.

We are now ready to state the main result of the paper --- a sharp comparison theorem for inradii of two-dimensional $\lambda$-convex domains. 

\begin{theorem}[Inradius comparison for isoperimetric $\lambda$-convex domains]
\label{inradestthm}
Let $M$ be a two-dimensional $\CBB(\kappa)$ space, and $D \subset M$ be a closed topological disk of inradius $r$ and with the rectifiable boundary of length $L$. If $D$ is $\lambda$-convex, then
\begin{equation}
\label{inradest}
r \geqslant \rho_\lambda(L).
\end{equation}
Moreover, equality holds if and only if $D$ is isometric to a $\lambda$-convex lune in $M^2(\kappa)$.
\end{theorem}

Computing explicitly the value of $\rho_\lambda(L)$ in each of the model ambient spaces, as a consequence of Theorem~\ref{inradestthm} we obtain the following sharp estimates for indradii of $\lambda$-convex domains in $\CBB(\kappa)$ spaces.

\begin{theorem}[Inradius lower bounds for $\lambda$-convex domains]
\label{estthm}
Let $M$ be a two-dimen\-sio\-nal $\CBB(\kappa)$ space, and $D \subset M$ be a $\lambda$-convex closed topological disk of inradius $r$ and of boundary length $L$. Then
\begin{enumerate}
\item
for $\kappa = k^2$ ($k>0$),
\begin{equation}
\label{Eq:sphere}
r \geqslant \frac{1}{k} \left(\arctan \frac{k}{\lambda} - \arctan \left(\frac{k}{\lambda} \cos \frac{L \sqrt{\lambda^2 + k^2}}{4}\right)\right);
\end{equation}

\item
for $\kappa = 0$,
\begin{equation}
\label{Eq:eucl}
r \geqslant \frac{1}{\lambda} \left(1 - \cos \frac{L\lambda}{4}\right);
\end{equation}

\item
\label{Eq:hyp}
for $\kappa = -k^2$ ($k>0$), depending on the value of $\lambda$:
\begin{enumerate}
\item
if $\lambda > k$,
\begin{equation}
\label{Eq:(a)}
r \geqslant \frac{1}{k} \left(\arctanh \frac{k}{\lambda} - \arctanh \left(\frac{k}{\lambda} \cos \frac{L \sqrt{\lambda^2 - k^2}}{4}\right)\right);
\end{equation}

\item
if $\lambda = k$,
\begin{equation}
\label{Eq:(b)}
r \geqslant \frac{1}{2k} \log \left(1 + \frac{L^2k^2}{16}\right);
\end{equation}

\item
and finally, if $k > \lambda > 0$,
\begin{equation}
\label{Eq:(c)}
r \geqslant \frac{1}{2k} \log \frac{(k+\lambda)\left(\cosh^2 \frac{L\sqrt{k^2 - \lambda^2}}{4} - \lambda^2/k^2\right)}{ (k - \lambda) \left(\cosh \frac{L\sqrt{k^2 - \lambda^2}}{4} + 1\right)^2}.
\end{equation}

\end{enumerate}
\end{enumerate} 
Moreover, equality holds if and only if $D$ is isometric to a $\lambda$-convex lune in $M^2(\kappa)$.
\end{theorem}

\begin{remark}
It is straightforward to check that in the hyperbolic case ($\kappa < 0$) the right-hand sides of inequalities (\ref{Eq:(a)})--(\ref{Eq:(c)}) have the `phase transition' behavior: for a fixed $k$, as $\lambda \to k+0$ the right side of (\ref{Eq:(a)}) tends to the right side of (\ref{Eq:(b)}), and as $\lambda \to k-0$ the right side of (\ref{Eq:(c)}) tends to the expression on the right side of (\ref{Eq:(b)}). Similarly, for a fixed $\lambda$, the right-hand side expression in (\ref{Eq:sphere}) tends to the right side in (\ref{Eq:eucl}) as $k \to 0+0$; the right side of (\ref{Eq:(a)}) tends to the right side expression in (\ref{Eq:eucl}) as $k \to 0-0$.
\end{remark}

\section*{Proof of the main results}

Let us start with stating some known results and proving some auxiliary lemmas all of which will be later used in the proof of the main result (Theorem~\ref{inradestthm}). 

In~\cite{BorNA}, Borisenko, combining celebrated results of Alexandrov \cite[p.\,269, p.\,318]{Al48} on gluing and isometric embedding of $\CBB(\kappa)$ spaces with Pogorelov's \cite[pp.\,119-167, p.\,267, pp.\,320-321]{Pog73} and Milka's \cite{Mil80} uniqueness results for such embeddings, proved the following theorem (which was not stated as such, but can be easily extracted from the proof of the main result of the paper):

\begin{theorem}[\cite{BorNA}]
\label{BorCBB}
Let $M^2$ be a two-dimensional $\CBB(\kappa)$ space, and $D \subset M^2$ be a $\lambda$-convex closed topological disk. Then there exists a non-closed convex surface $D_\kappa$ embedded into $M^3(\kappa)$ such that:
\begin{enumerate}
\item[1)]
$D_\kappa$ is isometric to $D$;

\item[2)]
the boundary $\partial D_\kappa$ is a closed curve lying in a two-dimensional totally geodesic subspace $\pi \simeq M^2(\kappa)$;

\item[3)]
$\partial D_\kappa$ is a $\lambda$-convex curve as a curve in $\pi$;

\item[4)]
$D_\kappa$ is a graph over $\pi$.\qed
\end{enumerate}
\end{theorem}

This theorem can be viewed as an extrinsic analog to celebrated Reshetnyak's Majorization Theorem in CAT($\kappa$) spaces (see~\cite{R68} and~\cite[Chapter 9]{BBI01}).

One of the classical results on the geometry of $\lambda$-convex domains is the following theorem, originally due to Blaschke \cite{Bla56} who proved it for smooth curves and surfaces in the Euclidean space. 

\begin{theorem}[Blaschke's Rolling Theorem, \cite{Bla56, Mi70}]
\label{Thm:BlaRoll}
If $D \subset M^2(\kappa)$ is a $\lambda$-convex domain, then 
\begin{equation*}
D \subseteq F_\lambda(s)
\end{equation*}
for every boundary point $s \in \partial D$; here $F_\lambda(s)$ is a closed convex set in $M^2(\kappa)$ such that its boundary curve is smooth, has constant curvature $\lambda$ and is locally supporting to $\partial D$ at $s$. \qed
\end{theorem}

\begin{remark}
As it follows from Theorem~\ref{Thm:BlaRoll} and the properties of $F_\lambda(s)$ (see the discussion after Definition~\ref{Def:LLune}), for $\kappa \geqslant 0$, or for $\kappa < 0$ and $\lambda > \sqrt{-\kappa}$ the domain $D$ is necessarily compact, while in the hyperbolic case ($\kappa < 0$) for $\lambda \in \left(0, \sqrt{-\kappa}\right]$ it may be unbounded. 
\end{remark}

For the proof of Theorem~\ref{inradestthm} we will need two technical lemmas. The first one will help us to reduce some of the later constructions to a more symmetric setup. In order to state the lemma, we need the following notation.

Suppose $D \subset M^2(\kappa)$ is a closed strictly convex set, and $\gamma$ is its boundary curve. Fix a point $o$ in the interior of $D$. For a unit direction $u$ in the tangent space $T_o M^2(\kappa)$ let $\sigma^+_u$ be the geodesic ray emanating from $o$ in the direction of $u$, and let $\sigma_u$ be the complete geodesic containing $\sigma^+_u$. Since $D$ is strictly convex, for every $u$ there exists a pair of geodesics supporting to $\gamma$ and perpendicular to $\sigma_u$. Among those there exists a unique geodesic $l(u)$ such that $\sigma^+_u$ points in the halfspace (with respect to $l(u)$) that does not intersect the interior of $D$. Conversely, every supporting geodesic to $\gamma$ defines a unique direction $u$ in the tangent space $T_o M^2(\kappa)$. Again by convexity of $D$, the mapping $\mathbb S^1 \ni u \mapsto l(u)$ is continuous (here we identify $\mathbb S^1$ and the set $T_o^1 M^2(\kappa)$ of all unit directions in $T_o M^2(\kappa)$).

For a given point $x \in \gamma$, let $u_x$ be a unit direction in $T_o M^2(\kappa)$ such that $l(u_x)$ is supporting to $\gamma$ at $x$. Note that if $\gamma$ has a unique tangent at $x$, then $u_x$ is uniquely defined by $x$; otherwise, there is a cone of directions. If there is a unique tangent, to simplify notation we will write $l_x := l(u_x)$. 

Let us pick a pair of points $x, y \in \gamma$ and a pair of directions $u_x, u_y \in \mathbb S^1 = T^1_o M^2(\kappa)$. The geodesic segment $xy$ cuts $D$ into two convex subregions, say $D_1$ and $D_2$. Denote $\alpha(u_x)$ (respectively, $\alpha(u_y)$) the angle between the geodesics $xy$ and $l(u_x)$ (resp.\,$xy$ and $l(u_y)$) measured from the side of $D_1$ (resp.\,$D_2$). Similarly, denote $\beta(u_y)$, respectively $\beta(u_x)$, the angle between $xy$ and $l(u_y)$, resp.\,$xy$ and $l(u_x)$, measured from the side of $D_1$, resp.\,$D_2$. By construction,
\begin{equation}
\label{Eq:InitAngleCond}
\alpha(u_x) + \beta(u_x) = \alpha(u_y) + \beta(u_y) = \pi.
\end{equation}

In the introduced notation the following lemma holds true.

\begin{lemma}
\label{ivthm}
For every closed strictly convex set $D \subset M^2(\kappa)$ and every point $o$ in the interior of $D$ there exist a pair of points $p^*, q^* \in \partial D$ and a pair of directions $u_{p^*}, u_{q^*} \in \mathbb S^1 = T_o^1 M^2(\kappa)$ such that $p^*$ and $q^*$ subdivide $\partial D$ into two arcs of equal length and 
\begin{equation}
\label{Eq:AngleCond}
\alpha(u_{p^*})+ \beta(u_{q^*}) = \alpha(u_{q^*}) + \beta(u_{p^*}) = \pi.
\end{equation}
\end{lemma}

\begin{proof}

The claim of the lemma seems to be folklore and follows from the intermediate value theorem by a standard trick. Since we were unable to locate a precise reference, we outline the proof here. 

First observe that because of (\ref{Eq:InitAngleCond}), if we satisfy $\alpha(u_{p^*}) + \beta(u_{q^*}) = \pi$, we automatically satisfy $\alpha(u_{q^*}) + \beta(u_{p^*}) = \pi$, and hence (\ref{Eq:AngleCond}). 

As usual, write $\gamma = \partial D$. We will prove Lemma~\ref{ivthm} for the case when $\gamma$ has a unique tangent geodesic at every point. The general case will then follow by approximation of convex sets by convex sets with smooth boundary. 

If $\gamma$ has the tangent geodesic at every point, then the map $x \mapsto u_x$ is well-defined and bijective. Therefore, the continuous mapping $\mathbb S^1 \ni u \mapsto l(u)$ descends to the continuous mapping $\gamma \ni x \mapsto l(u_x) =: l_x$.

Now let $x \in \gamma$ be a point. Define $f(x) \in \gamma$ to be the point such that $x$ and $f(x)$ subdivide $\gamma$ into two arcs of equal length. Clearly, the map $\gamma \ni x \mapsto f(x)$ is continuous. Let $\alpha_x := \alpha(u_x)$, $\alpha_{f(x)} := \alpha(u_{f(x)})$ (similarly define $\beta_x$ and $\beta_{f(x)}$) be the angles defined for the geodesic segment $x f(x)$ and the tangent geodesics $l_x$ and $l_{f(x)}$ (see the paragraph before Lemma~\ref{ivthm}). Because the assignments $\gamma \ni x \mapsto l_x$ and $\gamma \ni x \mapsto f(x)$ are both continuous, the map $\gamma \ni x \mapsto \alpha_x + \beta_{f(x)} - \pi$ is also continuous. Call this map $g$. Observe that $f$ is an involution, i.e. $f(f(x)) = x$; together with (\ref{Eq:InitAngleCond}) this implies
\begin{equation}
\label{Eq:gprop}
g\left(f(x)\right) = - g(x) \text{ for all }x \in \gamma.
\end{equation}

Pick a point $x_0 \in \gamma$. If $g(x_0) = 0$, then $p^* = x_0$ and $q^* = f(x_0)$ and we are done. If not, then without loss of generality assume $g(x_0) > 0$. But then $g(f(x_0)) < 0$ by (\ref{Eq:gprop}). Hence, since $g$ is continuous on an arc of $\gamma$ between the points $x_0$ and $f(x_0)$, by the intermediate value theorem there exists a point $p^*$ such that $g(p^*) = 0$. Then $p^*$ and $q^* = f(p^*)$ are the required points.
\end{proof}

Recall that $\rho_\lambda(L)$ is the inradius of a $\lambda$-convex lune of length $L$ (see the discussion before the statement of Theorem~\ref{inradestthm}). The following lemma describes some analytic properties of the function $\rho_\lambda$.

\begin{lemma}
\label{convlem}
The function $I_\lambda \ni L \mapsto \rho_\lambda(L)$ is smooth in the interior of $I_\lambda$, strictly increasing on $I_\lambda$, and for a given $L$ the value $\rho_\lambda(L)$ is equal to the corresponding right-hand sides in inequalities (\ref{Eq:sphere})--(\ref{Eq:(c)}).
\end{lemma}
\begin{proof}
The explicit formula for $\rho_\lambda(L)$ can be obtained by a direct and straightforward computation, and thus omitted (see \cite[Lemma 4.1]{DrPhD} for a similar computation). The monotonicity property is easy to establish by checking $d \rho_\lambda / d L (L)> 0$ for all values of $L$ in the interior of $I_\lambda$; this simple computation is also omitted.  
\end{proof}

\begin{proof}[Proof of Theorem~\ref{inradestthm}]

A) We start by proving the result assuming $D$ lies in the constant curvature space $M^2(\kappa)$. 

For simplicity, write $\gamma := \partial D$. Let $p^*, q^* \in \gamma$ be a pair of points given by Lemma \ref{ivthm}. In particular, $p^*$ and $q^*$ subdivide $\gamma$ into two arcs, say $\gamma_1$ and $\gamma_2$, of length $L/2$.

Suppose $m$ is the midpoint of the geodesic segment $p^*q^*$.  Let us show that
\begin{equation}
\label{goal1}
|ms| \geqslant \rho_\lambda(L) \quad \text{for every} \quad s \in \gamma.
\end{equation}
(Here $|ms|=|ms|_M$, and we will drop the index when it is clear in which space we measure the distance.)
Since the inradius is the radius of a largest ball contained in $D$, inequality~(\ref{goal1}) will yield~(\ref{inradest}).

Without loss of generality assume $s \in \gamma_1$. Let $R_m \colon M^2(\kappa) \to M^2(\kappa)$ be the point reflection in the point $m$. In other words, $R_m$ is an isometry of $M^2(\kappa)$ such that for every point $x \in M^2(\kappa)$ the point $m$ is the midpoint of the geodesic segment $x R_m(x)$. Let $\gamma_1' := R_m(\gamma_1)$ be the image of $\gamma_1$ under $R_m$, and $s' := R_m(s)$ be the image of $s$. By the choice of the segment $p^* q^*$ (see (\ref{Eq:AngleCond}) in Lemma~\ref{ivthm}), the curve $\Gamma_1 := \gamma_1 \cup \gamma_1'$ is convex and encloses a $\lambda$-convex domain (see Figure~\ref{Fig:Reflection}). By construction, the length of $\Gamma_1$ equals twice the length of $\gamma_1$, and hence is equal to $L$.

\begin{figure}[ht]
\begin{center}
\includegraphics[scale=.6, trim=30 10 10 10]{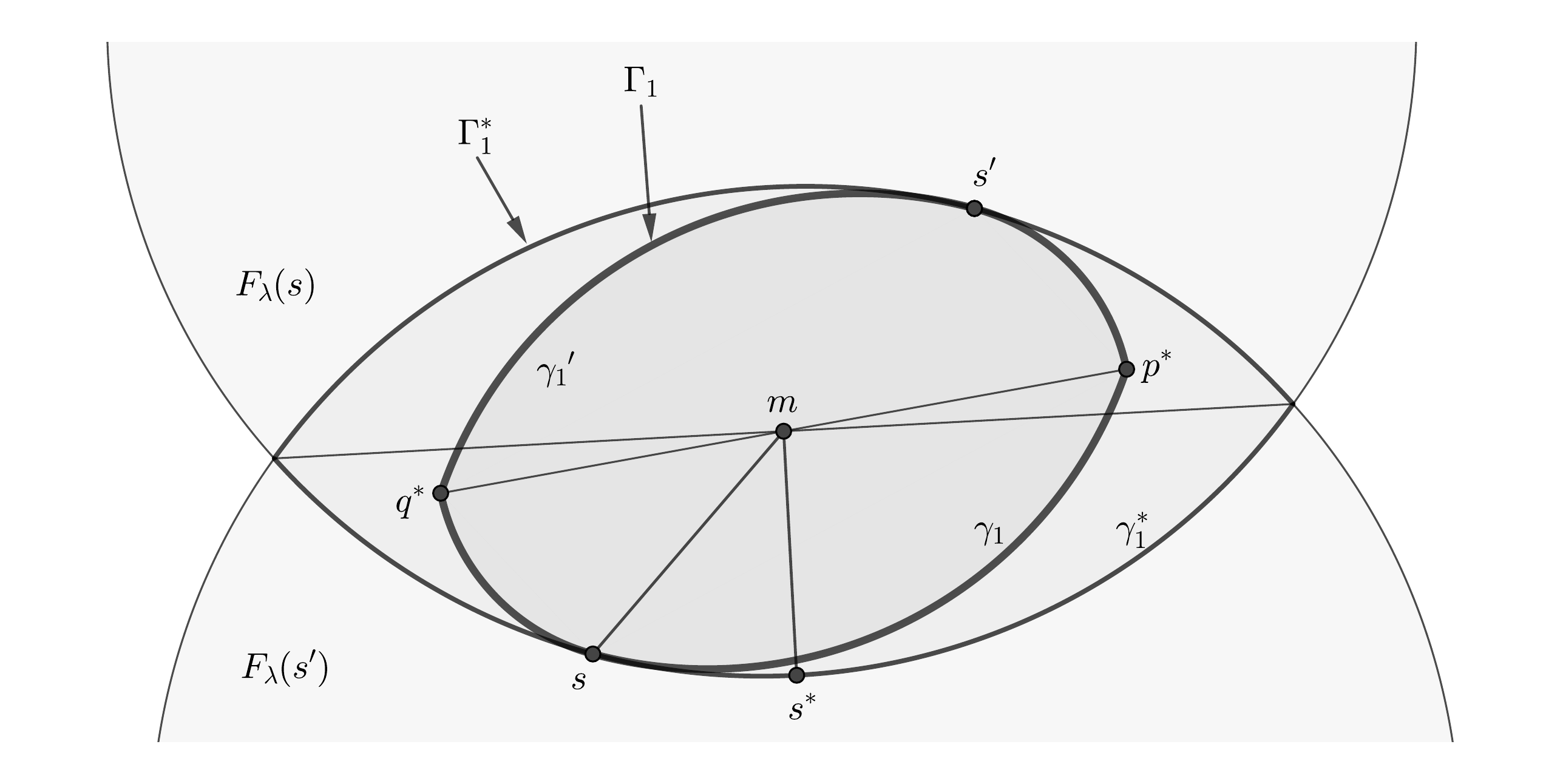}
\caption{The construction in the proof of Theorem~\ref{inradestthm} for the case when the ambient metric space is one of the constant curvature spaces. The subarc $\gamma_1$ of the curve $\gamma$ cut by the geodesic $p^*q^*$ given by Lemma~\ref{ivthm}. The curve $\gamma_1'$ is the image under the point reflection of $\gamma_1$ in $m$, where $m$ is the midpoint of the geodesic segment $p^*q^*$. The convex region bounded by the curve $\Gamma_1 = \gamma_1 \cup \gamma_1'$ is contained in the $\lambda$-convex lune $F_\lambda(s) \cap F_\lambda(s')$; here $\Gamma_1^*$ is the boundary curve of this lune.}
\label{Fig:Reflection}
\end{center}
\end{figure}

By Blaschke's Rolling Theorem (Theorem~\ref{Thm:BlaRoll}), the region bounded by $\Gamma_1$ is simultaneously contained in both $F_\lambda(s)$ and $F_\lambda(s')$ (see Theorem~\ref{Thm:BlaRoll} for the definition of sets $F_\lambda(s)$ and $F_\lambda(s')$). Therefore, $\Gamma_1$ is contained in the $\lambda$-convex lune $F_\lambda(s) \cap F_\lambda(s')$. Moreover, since $\Gamma_1$ is centrally symmetric with respect to $m$, we conclude $R_m\left(F_\lambda(s)\right) = F_\lambda(s')$, and hence $m$ is the center of the lune $F_\lambda(s) \cap F_\lambda(s')$. 

Write $\Gamma_1^*$ for the boundary of $F_\lambda(s) \cap F_\lambda(s')$. Suppose the length of $\Gamma_1^*$ is equal to $L^*$, then from the inclusion $\Gamma_1 \subset F_\lambda(s) \cap F_\lambda(s')$ we obtain
\[
L^\ast \geqslant L.
\]

Among the two curves of constant geodesic curvature composing $\Gamma_1^*$ consider the one, say $\gamma_1^*$, containing $s$, and let $s^*$ be the midpoint of $\gamma_1^*$ with respect to the arc length parameter on it. Observe that the geodesic $s^* m$ is the symmetry axis of $\Gamma_1^*$ (see Figure~\ref{Fig:Reflection}). 

We claim that 
\begin{equation}
\label{hypissue}
|mx| \geqslant |ms^*| \text{ for every } x \in \gamma_1^*,
\end{equation}
and inequality~(\ref{hypissue}) is sharp, unless $x = s^*$, or $m$ coincides with the center of $\gamma_1^*$ (which is only possible for either $\kappa \geqslant 0$, or $\kappa < 0$ and $\lambda > \sqrt{-\kappa}$, that is when $\gamma_1^*$ is a circular arc). The claim and the reflection symmetry of $\Gamma_1^*$ would imply, in particular, that $\rho_\lambda(L^*) = |ms^*|$.

In order to conclude (\ref{hypissue}) we will use the following geometric argument. First assume that we are in the hyperbolic case ($\kappa < 0$) and suppose $\lambda \leqslant \sqrt{-\kappa}$. Since the geodesic $s^*m$ is a symmetry axis of $F_\lambda(s)$, the distance function $\dist(m, \cdot)$ restricted on $\partial F_\lambda(s)$ is strictly increasing as we move along the curve $\partial F_\lambda(s)$ by increasing the arc length distance to $s^*$. For example, it can be verified in the Poincar\'{e} disk model: a hyperbolic circle centered at $m$ is disjoint from $\partial F_\lambda(s)$ if its radius is less than $|ms^*|$, touches $\partial F_\lambda(s)$ at $s^*$ if its radius is exactly $|ms^*|$, and intersects $\partial F_\lambda(s)$ in a pair of points symmetric with respect to $s^*m$ if the radius is greater than $m s^*$. Since hyperbolic circles and the curve $\partial F_\lambda(s)$ are Euclidean circles in this model, the monotonicity is easily observed from this intersection property. Therefore, (\ref{hypissue}) holds true in the case $\kappa < 0$ and $\lambda \leqslant \sqrt{-\kappa}$. A similar argument applies in the case $\kappa \geqslant 0$, or $\kappa < 0$ and $\lambda > \sqrt{-\kappa}$, that is when $F_\lambda(s)$ is a geodesic disk. For this we need to observe that any two circles in $M^2(\kappa)$ intersect in at most two points, and that $m$ necessarily lies on the geodesic segment connecting $s^*$ and the center of $F_\lambda(s)$.

In particular, (\ref{hypissue}) holds for $x = s$. Therefore,
\begin{equation}
\label{est1}
|ms| \geqslant |m s^*| = \rho_\lambda(L^*) \geqslant \rho_\lambda(L),
\end{equation}
where we used monotonicity of $\rho_\lambda(\cdot)$ (Lemma~\ref{convlem}). This finishes the proof of inequality in Theorem~\ref{inradestthm}. 

Let us analyze the equality case. If equality is attained in (\ref{inradest}), then there must be at least a pair of points $s_1 \in \gamma_1$ and $s_2 \in \gamma_2$ for which in (\ref{goal1}) we have equality. Indeed, existence of at least one such point, say $s_1$ on $\gamma_1$, follows directly from the equality assumption in (\ref{inradest}). Note that if $s_1$ is one of the endpoints of $\gamma_1$, we are done; so we can assume that at the endpoints of $\gamma_1$, and thus of $\gamma_2$, inequality (\ref{goal1}) is strict. If $|ms| > \rho_\lambda(L)$ for all $s\in \gamma_2$, then $\dist(m, \gamma_2) > \rho_\lambda(L)$ while $\dist(m, \gamma_1) = \rho_\lambda(L)$. Hence the disk of radius $\rho_\lambda(L)$ centered at $m$ can be moved in the direction orthogonal to $p^*q^*$ so that it will be entirely in the interior of $D$. That would mean that the inradius of $D$ is strictly larger than $\rho_\lambda(L)$, which is in contradiction to the equality assumption in (\ref{inradest}). 

Consider the point $s_1 \in \gamma_1$. Then (\ref{est1}) for $s = s_1$ implies $\rho_\lambda(L^*) = \rho_\lambda(L)$, and since $\rho_\lambda(\cdot)$ is strictly increasing (see Lemma~\ref{convlem}), this is only  possible when $L = L^*$. But, unless $\Gamma_1 = \Gamma_1^*$, the inclusion of the convex regions enclosed by these curves implies strict inequality between their lengths. Hence, necessarily $\Gamma_1 = \Gamma_1^*$ and $\gamma_1 = \gamma_1^*$. The same flow of arguments applied to $s_2 \in \gamma_2$ leads to the conclusion $\gamma_2 = \gamma_2^*$ (the curves $\gamma_2$ and $\gamma_2^*$ are defined analogously to $\gamma_1$ and $\gamma_1^*$). Therefore, equality in (\ref{inradest}) is only possible when $D$ is a $\lambda$-convex lune. Theorem~\ref{inradestthm} for $M = M^2(\kappa)$ is proven.

B) Let us move to the general case. By Theorem~\ref{BorCBB}, consider a convex cap $D_\kappa \subset M^3(\kappa)$ isometric to $D$ with the planar boundary $\partial D_\kappa \subset M^2(\kappa) \simeq \pi \subset M^3(\kappa)$. Suppose $D_\kappa^*$ is the domain enclosed by $\partial D_\kappa$ in the two-dimensional totally geodesic plane $\pi \simeq M^2(\kappa)$ (planarity is guaranteed by property 2) in Theorem~\ref{BorCBB}). If $r^*$ is the inradius of $D_\kappa^*$, then by the graph property of $D_\kappa$ (see property 4) in Theorem~\ref{BorCBB}) we have
\begin{equation}
\label{ineq1}
r \geqslant r^*.
\end{equation}
Indeed, let $\omega^*$ be an inscribed circle for $D^*_\kappa$, and $p^*$ be the center of $\omega^*$. Consider a curve $\omega \subset D_\kappa$ and a point $p \in D_\kappa$ that are projected orthogonally onto $\omega^*$ and $p^*$, respectively. Then for every $q \in \omega$ and the corresponding $q^* \in \omega^*$ we have
\[
|pq|_{D_\kappa} \geqslant |p^*q^*|_{M^3(\kappa)} = |p^* q^*|_{\pi}= r^*
\]
(where $|\cdot|_{D_\kappa}$ is the distance between points measured intrinsically in $D_\kappa$). 
Therefore, the cap $D_\kappa$ contains an intrinsic ball of radius $r^*$ and with the center at $p$, and hence the inradius of $D_\kappa$ cannot be smaller than $r^*$. This proves~(\ref{ineq1}).

By properties 1) and 3) of Theorem~\ref{BorCBB}, $D_\kappa^*$ is a $\lambda$-convex domain with the boundary curve of length $L$. Thus we are in position of applying the result of part A) to this domain; we obtain
\begin{equation}
\label{ineq2}
r^* \geqslant \rho_\lambda (L).
\end{equation}
Inequalities (\ref{ineq1}) and (\ref{ineq2}) together imply estimate (\ref{inradest}) in Theorem~\ref{inradestthm} for CBB($\kappa$) spaces. Moreover, in order to have equality in~(\ref{inradest}), one should have equalities in both~(\ref{ineq1}) and~(\ref{ineq2}). Equality in~(\ref{ineq1}) is attained if $D_\kappa$ coincides with $D_\kappa^*$. Equality in~(\ref{ineq2}) implies, by the consideration in A), that $D_\kappa^*$ is a $\lambda$-convex lune. Hence, if equality in (\ref{inradest}) is attained, $D$ is necessarily isometric to a $\lambda$-convex lune in $M^2(\kappa)$. Converse is obvious. Theorem~\ref{inradestthm} is proven.
\end{proof}

\begin{proof}[Proof of equivalence of Theorem~\ref{estthm} and Theorem~\ref{inradestthm}]

Theorem~\ref{estthm} follows immediately from Theorem~\ref{inradestthm} and Lemma~\ref{convlem}. Conversely, Theorem~\ref{estthm} and Lemma~\ref{convlem} imply Theorem~\ref{inradestthm}.
\end{proof}

\section*{Concluding remarks and open questions}

We proved that among all $\lambda$-convex domains of the same perimeter $\lambda$-convex lunes have the smallest inradius. A similar question can be asked: find among all $\lambda$-convex domains of a given \textit{area} those having the smallest inradius. Such domains necessarily exist, and we pose the following conjecture.

\begin{conjecture}
\label{Conj}
In two-dimensional $\CBB(\kappa)$ spaces, among all $\lambda$-convex domains (homeomorphic to a disk) of a given area $\lambda$-convex lunes are the only ones with the smallest inradius.
\end{conjecture} 

It is not hard to prove this conjecture for constant curvature spaces $M^2(\kappa)$ (essentially, by using the same techniques as were used in this paper). However, the problem of extending this to general metric spaces (in particular, what is an analog of Theorem~\ref{BorCBB} that reduces questions about equal-area domains to model spaces?) remains unsolved.

Parallel to an inradius we can define a \textit{circumscribed radius}
\[
R:= \min_{p \in D} \max_{q \in \partial D} |pq|_M
\]
for a closed topological disk $D$ in a metric space $M$. We conjecture that $\lambda$-convex lunes are the only \textit{maximizers} of the circumscribed radius among all isoperimetric $\lambda$-convex domains. This is certainly true in $M^2(\kappa)$; an extension to metric spaces of bounded curvature is an open question.

Finally, all questions mentioned in this paper in a two-dimensional setting make sense, and hence give rise to unsolved problems, for multidimensional $\lambda$-convex domains (see~\cite{Bor02, BorDr13} for the $n$-dimensional setup).

\section*{Acknowledgments}
The author would like to thank Alexander Borisenko for looking through one of the preliminary versions of the manuscript and providing useful remarks. We are also grateful to the anonymous referee for the suggestions that helped to improve the exposition.

The research was partially supported by the Akhiezer Foundation (Kharkiv, Ukraine) and by the advanced grant ``HOLOGRAM'' of the European Research Council (ERC).


\end{document}